\documentclass[reqno,11pt]{amsart}

\usepackage{amssymb,amsthm}
\usepackage{amsmath}
\usepackage{amsfonts}
\usepackage{dsfont}

\usepackage[a4paper,  margin=3.3cm]{geometry}
\usepackage[active]{srcltx}
\makeatletter\@addtoreset{equation}{section}\makeatother

\newtheorem{theorem}{Theorem}[section]
\newtheorem{corollary}[theorem]{Corollary}
\newtheorem{lemma}[theorem]{Lemma}

\newtheorem{proposition}[theorem]{Proposition}

\newtheorem{definition}[theorem]{Definition}

\numberwithin{equation}{section}

\title[Lusin spaces as images of locally compact Polish spaces]{Lusin spaces as images of locally compact Polish spaces}
\author{Alina Kargol}

\address{Instytut Matematyki, Uniwersytet Marii Curie-Sk{\l}odowskiej, plac Marii Curie-sk{\l}odowskiej 1, 20-031 Lublin, Poland; fax +48 81 537 5471}
\email{alina.kargol@mail.umcs.pl}

\author{ Yuri  Kozitsky}

\address{Instytut Matematyki, Uniwersytet Marii Curie-Sk{\l}odowskiej, plac Marii Curie-sk{\l}odowskiej 1, 20-031 Lublin, Poland; fax +48 81 537 5471}
\email{jurij.kozicki@mail.umcs.pl}

\begin{document}


\begin{abstract}
A Lusin space is a Hausdorff space being the image of a Polish space
under a continuous bijection. Such spaces have multiple
applications, in particular, as state spaces of various stochastic
systems. In this work, we consider the spaces obtained as the images
of a noncompact and locally compact Polish space $(X, \mathcal{T})$,
which we call $c$-Lusin. The main result is the statement that a
$c$-Lusin space $Y=f(X)$, can be written as $Z\cup Y_1$, where $Z$
is a locally compact Polish space whereas $Y_1$ is $c$-Lusin. At the
same time, $Y_1$ is the set of the discontinuity points of $f^{-1}$
which is a closed subset of $Y$. Moreover, $Y_1$ is nowhere dense if
 (and only if) $Y$ is a Baire space.
 By the same arguments, $Y_1$ can also be
decomposed as $Z_1 \cup Y_2$ with the properties as above. In the
case where $f$ can be extended to a continuous map $f:X\cup
\{\infty\} \to Y$, and thus $Y_1$ is a singleton, we  construct a
metric on $X$ such that the corresponding metric space is compact
and homeomorphic to the $c$-Lusin space $(f(X), \mathcal{T}')$.
\end{abstract}

\maketitle

\section{Introduction}

A Polish space is a separable space  the topology of which is
consistent with a complete metric. Polish spaces, as well as their
images, cf.  \cite[pages 239--277]{Cohn} and \cite{Levi,Nik}, have
multiple applications. A Lusin space is a Hausdorff space that is
the image of a Polish space under a continuous bijection, see
\cite{Bes} or \cite[page 273]{Cohn}.

To have the freedom of dealing with different topologies defined on
the same underlying set $X$, we use the notation $(X,\mathcal{T})$
for the corresponding topological space.  Let $(X,\mathcal{T})$ and
$(Y,\mathcal{T}')$ be a Polish space and $Y=f(X)$ for a continuous
bijection $f$. An important feature of this pair
 is that the spaces
are Borel isomorphic, see \cite[Proposition 8.6.13, page 275]{Cohn}.
In view of this fact, along with Polish spaces Lusin spaces are
frequently used as state spaces of a broad range of stochastic
systems, see, e.g., \cite{BBR,BK,Craul,Dawson,PT}.

Lusin spaces have a plenty of topological properties, see e.g.,
\cite{Bes} and the literature quoted therein. Some of them may be
quite different from those of Polish spaces. In particular, a Lusin
space need not be metrizable. An example can be an infinite
dimensional linear space, $X$, equipped with the weak topology
relative to the topology that makes $X$ a separable Banach space.
The situation with the mentioned properties can get more
controllable if one assumes additional properties of the Polish
space $(X,\mathcal{T})$. If it is compact, then  $(f(X),
\mathcal{T}')$ is also compact and $f$ is a homeomorphism. In this
work, we study the spaces obtained as the images of a noncompact and
locally compact Polish space $(X,\mathcal{T})$. For convenience, we
call them $c$-\emph{Lusin} spaces. Our main result is the following
statement, see Theorem \ref{1tm} below.  Let $Y_1$ be the set of all
discontinuity points of $f^{-1}$. Then $Y_1$ is a closed subset of
$Y=f(X)$, which is a $c$-Lusin space in the subspace topology --
since its preimage $X_1$ is a locally compact Polish space (by the
continuity of $f$). At the same time, $Z:=Y \setminus Y_1$ is a
locally compact Polish space for it is homeomorphic to its preimage
$W$. Moreover, $Y_1$ is nowhere dense (hence $Z$ is dense) in $Y$ if
and only if $Y$ is a Baire space. If $f$ is \emph{feebly} open,
which implies that $Y$ is Baire, then also $X_1$ is a nowhere dense
subset of $X$. By repeating the same arguments one obtains the
decomposition $Y_1 = Z_1 \cup Y_2$, where $Z_1$ is a locally compact
Polish space and $Y_2$ is $c$-Lusin. This procedure can be continued
ad infinitum, or to the step at which $Z_k = Y_k$ or $Z_k =
\varnothing$, hence $Y_{k+1}=Y_k$. In Sect. 3, we study the case
where $f$ can be continuously extended to the Alexandroff
compactification $X\cup \{\infty\}$, and thus $Y_1$ is at most
singleton. For $Y_1=\{y_0\}$, hence $X_1 =\{x_0\}$,  $x_0 =
f^{-1}(y_0)$, in Theorem \ref{2tm} we show that $(Y, \mathcal{T}')$
is a compact Polish space homeomorphic to $(X, \mathcal{T}_{x_0})$,
where  $\mathcal{T}_{x_0}$ is the metric topology corresponding to
an explicitly constructed complete $x_0$-dependent metric.

\section{The main result}

We begin by making precise the notions used throughout. A set the
closure of which has empty interior is called nowhere dense.
\begin{definition}
  \label{1df}
A topological space is called a Baire space if the  union of any
countable collection of its closed subsets each with empty interior
also has empty interior.
\end{definition}
A detailed presentation of the properties of Baire spaces can be
found in \cite{HM}. Among them there are the following ones: (a)
every nonempty open subset of a Baire space is of second category
\cite[page 11]{HM}, and hence is a Baire space in the subspace
topology; (b) closed subsets need not be Baire; (c)  both locally
compact Hausdorff and completely metrizable spaces are Baire, see
\cite[Theorems 2.3 and 2.4]{HM}. Thus, our Polish space
$(X,\mathcal{T})$ is a Baire space. However, its continuous image
$Y=f(X)$ need not be such, which can be seen from the following
example. Consider $(\mathds{Q}, \mathcal{T})$ and $(\mathds{Q},
\mathcal{T}_{|\cdot|})$, where $\mathds{Q}$ is the set of rational
numbers, $\mathcal{T}=2^\mathds{Q}$, i.e., contains all subsets of
$\mathds{Q}$, and $\mathcal{T}_{|\cdot|}$ is the norm topology
related to absolute value $|\cdot|$. The metric $d(x,y) = 1$ for
$x\neq y$, and $d(x,y)=0$ for $x=y$, is consistent with
$\mathcal{T}$; hence, the former is a Polish space, which is
obviously locally compact. Since the embedding map $f:(\mathds{Q},
\mathcal{T})\to (\mathds{Q},\mathcal{T}_{|\cdot|})$ is continuous,
the latter is a $c$-Lusin space, which is apparently not Baire.
Noteworthy, these spaces are Borel isomorphic as their Borel
$\sigma$-fields coincide with $2^\mathds{Q}$ in this case.

The following is known.
\begin{proposition}\cite[Theorem 1.15]{HM}
  \label{0pn}
Every space which contains a dense Baire subspace is a Baire space.
\end{proposition}
A map $f:(X, \mathcal{T}) \to (Y, \mathcal{T}')$ is called
\emph{feebly} open if for each nonempty $A\in \mathcal{T}$, there
exists a nonempty $B\in \mathcal{T}'$ such that $B\subset f(A)$.
\begin{proposition}\cite[Theorem 4.1]{HM}
  \label{1pn}
Let $(X, \mathcal{T})$ be a Baire space and $f:(X, \mathcal{T}) \to
(Y, \mathcal{T}')$ a be continuous and feebly open surjection. Then
$(Y, \mathcal{T}')$ is also a Baire space.
\end{proposition}
 Our main result is the following statement.
\begin{theorem}
\label{1tm} Let $(Y, \mathcal{T}')$ be a $c$-Lusin space obtained by
means of a continuous bijection $f:(X, \mathcal{T})\to (Y,
\mathcal{T}')$. Then there exists an open subset $Z\subset Y$ such
that: (a) $(Z, \mathcal{T}'_Z)$ is a locally compact Polish space;
(b) $(Y_1 , \mathcal{T}'_1)$ is a $c$-Lusin space. Here
$\mathcal{T}'_Z = \{ A\cap Z : A \in \mathcal{T}'\}$, $Y_1 =
Y\setminus Z$ and $\mathcal{T}'_1 = \{ A\cap Y_1 : A \in
\mathcal{T}'\}$. $Y_1$ is a nowhere dense subset of $Y$ if and only
if $(Y, \mathcal{T}')$ is a Baire space.  If $f$ is feebly open,
hence $Y$ is Baire, then $X_1:=f^{-1}(Y_1)$ is nowhere dense in $X$.
\end{theorem}
\begin{proof}
Let $(X, \mathcal{T})$ and $Y=f(X)$  be as assumed. For $y\in Y$, by
$\mathcal{T}'(y)$ we denote the set of all $C\in \mathcal{T}'$ such
that $y\in C$.  For $x\in X$, $\mathcal{T}(x)$ is defined
analogously.  Now we define
\begin{equation}\label{f1}
Z=\{ y\in Y: \exists C\in \mathcal{T}'(y) \ f^{-1}(C) \ {\rm has} \
{\rm compact} \ {\rm closure}\}.
\end{equation}
Clearly, $Z\in \mathcal{T}'$ and $f^{-1}$ is discontinuous at each
$y\in Y_1 := Y \setminus Z$ since each $x=f^{-1}(y)$, also for $y\in
Y_1$, has a compact neighborhood. Let us show now  that $f^{-1}$ is
continuous at each $z\in Z$. Consider a net $\{y_\kappa\}\subset Y$
convergent to a given $z\in Z$, and set $x_\kappa =
f^{-1}(y_\kappa)$. Let also $C\in \mathcal{T}'(z)$ be as in
(\ref{f1}), and then $K\subset X$ be the compact closure of
$f^{-1}(C)$. Then the preimage of the part $\{y_\kappa\}\cap C$ lies
in  $K$, and hence possesses accumulation points. By the continuity
of $f$ it follows that all of them should coincide with the preimage
of $z$, which yields the continuity in question.

Denote $X_1 = f^{-1}(Y_1)$ and $W= f^{-1}(Z)$, and also
$\mathcal{T}_W = \{A\cap W: A\in \mathcal{T}\}$,
$\mathcal{T}_1=\{A\cap X_1: A\in \mathcal{T}\}$. Since $W$ is open
and $X_1$ is closed, both $(W, \mathcal{T}_W)$ and $(X_1,
\mathcal{T}_1)$ are Polish spaces, see \cite[Proposition 8.1.2, page
240]{Cohn}. These spaces are also locally compact, see \cite[Theorem
3.3.8]{Eng}. Hence, $(Y_1, \mathcal{T}'_1)$ is a $c$-Lusin space,
which yields the validity of claim (b). At the same time, by the
continuity of $f^{-1}|_{Z}$ the restriction $f_W:= f|_W$ is a
homeomorphism between $(W,\mathcal{T}_W)$ and $(Z,\mathcal{T}_Z)$,
which means that the latter is a locally compact Polish space, which
proves claim (a).

Let $Y$ be a Baire space. Every locally compact and second countable
space is also $\sigma$-compact. That is, there exists a sequence of
compact subsets, $\{K_n\}_{n\in \mathds{N}}$, such that: (a) each
$K_n$ is contained in the interior of $K_{n+1}$; (b) $X =\cup_n
K_n$. For such a sequence, we set $Q_n = f(K_n)$ and let $C_n$ be
the interior of $Q_n$. Then $Y= \cup_n Q_n$. Since $Y$ is a Baire
space, $C_n\neq \varnothing$ for at least some of $n\in \mathds{N}$.
On the other hand, $C_n\subset Z$ for each $n$, see (\ref{f1}). Set
$Q_{n, 1} = Q_n \cap Y_1$, which means that $Y_1 = \cup_n Q_{n,1}$.
Let $C_{n,1}$ be the interior of $Q_{n,1}$. Then $C_{n,1}\subset C_n
\subset Z$; hence, $C_{n,1}=\varnothing$ for all $n$; thus, $Y_1$ is
nowhere dense as the closed set of first category, see Definition
\ref{1df}. On the other hand, if $Z$ is dense in $Y$, then $Y$ is a
Baire space -- by Proposition \ref{0pn} and the openness of $Z$
shown above.

Finally, let $f$ be feebly open. If the interior of $X_1$ is
nonempty, then the interior of $Y_1$ should also be nonempty which
is impossible in this case  by Proposition \ref{1pn}. Thus,  $X_1$
ia a nowhere dense  subset of $X$. This completes the whole proof.
\end{proof}
From the proof made above, it readily follows that $X_1$ can be
characterized by the following property. Let $X\cup \{\infty\}$ be
the Alexandroff compactification of $(X,\mathcal{T})$. Let also
$\mathfrak{X}$ stand for the set of all sequences $\{x_n\}\subset X$
that converge to $\infty$. That is, each  $\{x_n\}\in \mathfrak{X}$
is eventually in $X \setminus K$ for each compact $K$. Then
$X_1=f^{-1}(Y_1)$ can be written in the form
\begin{equation}
  \label{f3}
  X_1 =\{ x\in X: \exists \{x_n\}\in \mathfrak{X} \ \lim_{n\to +\infty} f(x_n) = f(x)\}.
\end{equation}
By (\ref{f3}) one can see that the structure of $X_1$ (and thus of
$Y_1$) is predetermined by the properties of $f$ in the vicinity of
$\infty$. In particular, $X_1$ is at most singleton if $f$ has a
continuous extension to $X\cup \{\infty\}$. In the aforementioned
example with $X=Y=\mathds{Q}$, $\{x_n\}\in \mathfrak{X}$ means that,
for some $k$, $x_n \neq x_m$ for $n, m\geq k$. Take any $x\in
\mathds{Q}$ and set $x_n = x + 1/n$. Then this sequence belongs to
$\mathfrak{X}$. At the same time, $|x_n- x|= 1/n$, which means that
$x\in X_1$, and hence $X_1=X$.

By repeating the arguments used in the proof of Theorem \ref{1tm} we
obtain the following statement that establishes the structure of the
$c$-Lusin space $(Y, \mathcal{T}')$.
\begin{corollary}
  \label{1co}
There exists a descending sequence $\{Y_k\}_{k\in \mathds{N}}$ of
closed subsets of $Y$ such that, for each $k\in \mathds{N}$,
$(Y_k,\mathcal{T}'_k)$ and $(Z_k , \mathcal{T}_{Z_k}')$ are a
$c$-Lusin space and a locally compact Polish space, respectively.
Here $\mathcal{T}'_k = \{A \cap Y_k: A\in \mathcal{T}'\}$ and
$\mathcal{T}'_{Z_k} = \{A \cap Z_k: A\in \mathcal{T}'\}$.
\end{corollary}
Note that the aforementioned sequence may end up with
$Y_{k+1}=\varnothing$ or $Y_{k+1}=Y_k$ for some $k\in \mathds{N}_0$.
Here by $Y_0$ and $Z_0$ we mean $Y$ and $Z$, respectively. For $A\in
\mathcal{T}'$, set $A_{k+1} = A\cap Z_k$, $k\in \mathds{N}_0$. Then
$A_k \cap A_{k'}=\varnothing$ for all distinct $k$ and $k'$.
Therefore,
\begin{equation*}
%\label{f4}
A = \bigcup_{k\in \mathds{N}} A_k
\end{equation*}
is a disjoint decomposition of $A$ in which $A_k \in
\mathcal{T}'_{Z_{k-1}}$, $k\in \mathds{N}$.

Now we introduce a complete metric $\delta_k$, consistent with
$\mathcal{T}_{Z_k}'$,  $k\geq 0$.  In this context, we also set $W_k
= f^{-1}(Z_k)$ and $X_k = f^{-1}(Y_k)$, $k\geq 0$. Let $d$ be a
complete metric consistent with $\mathcal{T}$. For $x\in X_k$ and
$y\in Z_k$, $k\in \mathds{N}_0$, we define
\begin{equation}\label{m1}
d(x, X_{k+1}) = \inf_{v\in X_{k+1}} d(x, v), \qquad  \varkappa_k (y)
= 1/ d(f^{-1}(y), X_{k+1}).
\end{equation}
That is, $d(f^{-1}(y), X_{k+1})$ is the distance from the pre-image
of $y$ to $X_{k+1} = X_k \setminus W_k$. Now similarly as in
\cite[page 240]{Cohn} we set
\begin{equation*}
%\label{f2}
\delta_k (y , y') = d(f^{-1}(y), f^{-1}(y')) + \left|\varkappa_k (y) - \varkappa_k(y') \right|, \qquad y,y'\in Z_k, \quad k\geq 0.
\end{equation*}
By this construction, it follows that
\begin{gather*}
  \delta_k (y , y') = d_k (x, x') := d(x,x') + \left|\frac{1}{d(x,X_{k+1})} - \frac{1}{d(x',X_{k+1})}   \right|,
\end{gather*}
where $x=f^{-1}(y)$, $x'=f^{-1}(y')$. One can show, see the proof of
Proposition 8.1.2 in \cite{Cohn}, that $d_k$ is a complete metric
consistent with $\mathcal{T}_{W_k}$. Then $\delta_k$ is the metric
in question.

\section{A special case}

Here we consider the case where $f$ has a continuous extension to
the Alexandroff compactification of $X$. As mentioned above, this
property corresponds to $X_1$ consisting of at most one element. If
$X_1=\varnothing$, then $f$ is a homeomorphism. Recall that $(X,
\mathcal{T})$ is noncompact.
\begin{proposition}
\label{1apn}
 $(Y,\mathcal{T}')$ is compact if $X_1$ is  a singleton.
\end{proposition}
\begin{proof}
Set $X_1 =\{x_0\}$. Let $\{y_\iota\}\subset Y$ be a net, for which
we have the corresponding net of $x_\iota = f^{-1}(y_\iota)$. Then
ether $\{x_\iota\}\subset K$ for some compact $K\subset X$, or it
contains a sub-net, $\{x_\kappa\}$, convergent to $\infty$. In the
former case, the net $\{y_\iota\}$ is contained in the compact
$f(K)$, and hence has accumulation points. Otherwise, the sub-net
$\{y_\kappa\}$ converges to $f(x_0)$, which yields the compactness
of  $(Y,\mathcal{T}')$.
\end{proof}
A priori, even being compact $(Y, \mathcal{T}')$ need not be
metrizable. In the next statement, we nevertheless show that it is.
\begin{theorem}
  \label{2tm}
Assume that $X_1 = \{x_0\}$, and hence $f$ has a continuous
extension to the $X\cup \{\infty\}$. Then there exists a metric,
$\delta$, on $X$ such that $(X,\delta)$ is a compact metric space
homeomorphic to $(Y,\mathcal{T}')$.
\end{theorem}
The proof will be done by an explicit construction of $\delta$. Its
main idea stems from  the proof of Proposition \ref{1apn}, by which
$(Y, \mathcal{T}')$ is a one-point compactification of $(Z,
\mathcal{T}'_Z)$. To figure it out, we take two (disjoint) sequences
$\{x_k\}, \{x'_k\}\subset X\cup \{\infty\}$ such that $x_k \to
\infty$ and $x'_k\to x_0$. Then the closures in $Y$ of their
$f$-images contain $f(x_0)$. This yields that the map $f^{-1}$ from
$Z$ to $X\cup\{\infty\}$ is not uniformly
$\mathcal{T}'/\mathcal{T}$-continuous. Hence, by Taimanov's theorem,
see \cite{Taim},  it does not have a continuous extension to $Y$ in
this case. At the same time, it may get such an extension if one
identifies $\infty$ and $x_0$, and thus their neighborhoods. Then
$\delta$ is obtained by applying the corresponding construction of
\cite{Mandel}, modified to take into account the mentioned
identification. We thus begin by making this step.

For a nonempty $D\subset X$ and $r>0$, we set $D^r=\{x\in X: \exists v\in D \ d(x,v)<r\}$, where $d$ is as in (\ref{m1}). Let $\{K_n\}\subset X$ be an ascending sequence of compact subsets that exhausts $X$ and is such that $K_{n}$ is contained in the interior of $K_{n+1}$, $n\in \mathds{N}$.  Then one finds $\{r_n\}_{n\in \mathds{N}}\subset (0,+\infty)$ such that: (a) $K^{r_n} \subset K_{n+1}$; (b) $r_n > r_{n+1}$; (c) $r_n \to 0$ as $n\to +\infty$.
Of course, we can also assume that $x_0\in K_1$. Let us now define the following functions
\begin{equation}
  \label{1}
  g(x) = \max_{n\in \mathds{N}} \left[r_n - d(x, K_n) \right], \qquad h(x) = \min\{ d(x,x_0) ; g(x)\}, \quad x\in X.
\end{equation}
Note that $h(x_0)=0$ and $g(x_0) = r_1>0$. Moreover, $g(x)>0$ for all $x$. Let us prove that
\begin{equation}
  \label{2}
  |h(x) - h(y)|\leq d(x,y), \qquad x,y \in X,
\end{equation}
i.e., $x\mapsto h(x)$ is Lipschitz-continuous. Since the right-hand side of (\ref{2}) is symmetric with respect to the interchange $x\leftrightarrow y$, it is enough to show that
\begin{equation}
  \label{3}
   h(x) - h(y)\leq d(x,y), \qquad x,y \in X.
\end{equation}
For $h(y)=d(y,x_0)$, (\ref{3}) holds true by the triangle inequality for $d$. Indeed, by (\ref{1}) we have
\begin{gather*}
 % \label{4}
  h(x) - d(y,x_0) \leq d(x,x_0) - d(y,x_0) \leq d(x,y).
\end{gather*}
Assume now that $h(y)=g(y)$. By (\ref{1}) we then have $h(x) - h(y) \leq g(x) - g(y)$, which means that (\ref{3}) will follow by
\begin{equation*}
  %\label{5}
  g(x) - g(y) \leq d(x,y).
\end{equation*}
First we consider the case of $g(x)=r_n$ for some $n\in \mathds{N}$, which corresponds to $x\in K_n$. For this $n$, by (\ref{1}) we have $g(y) \geq r_n - d(y, K_n)$, which implies
\[
g(x) - g(y) \leq d(y, K_n) \leq d(x,y).
\]
For $g(x)= r_n - d(x, K_n)$, similarly we have
\begin{gather*}
  %\label{6}
  g(x)-g(y) \leq d(y, K_n) - d(x,K_n) = d(y, K_n) - d(x,z) \\[.2cm] \nonumber
 \leq d(y, z) - d(x,z) \leq d(x,y),
\end{gather*}
where $z\in K_n$ is such that $d(x,K_n) = d(x,z)$. This completes the proof of (\ref{2}).

Let us now introduce a candidate for another metric on $X$. Set
\begin{equation}
  \label{7}
  \delta(x,y) = \min\{d(x,y); h(x) + h(y)\}.
\end{equation}
\begin{lemma}
  \label{1lm}
It follows that $\delta$ defined in (\ref{7}) is a metric, such that the metric space $(X,\rho)$ is compact.
\end{lemma}
\begin{proof}
Obviously $\delta$ is  symmetric and $\delta(x,y)=0$ implies $x=y$. Then to prove that $\delta$ is a metric it remains to show that
\begin{equation}\label{8}
\delta(x,y) \leq \delta(x,z) + \delta(y,z), \qquad x,y,z \in X.
\end{equation}
By (\ref{7}), it follows that each $\delta$ in (\ref{8}) can equal either the corresponding $d$ or the sum of two corresponding $h$.
If all the three $\delta$'s equal the corresponding $d$'s, then (\ref{8}) follows by the triangle inequality for $d$. Let us consider the case where one of the $\delta$'s in (\ref{8}) equals the sum of the $h$'s, whereas the remaining  two equal the corresponding $d$'s. If this holds on the left-hand side -- which by (\ref{7}) corresponds to $h(x) + h(y)\leq d(x,y)$ --  then (\ref{8}) turns into
\begin{equation*}
%  \label{9}
  h(x) + h(y) \leq d(x,z) + d(y,z).
\end{equation*}
As just mentioned, $h(x) + h(y) \leq d(x,y)$, which yields the
validity of (\ref{8}) by the triangle inequality for $d$. Now let
the mentioned equality holds on the right-hand side of (\ref{8}). By
the symmetry $x\leftrightarrow y$, it is enough to consider only the
case $\delta(y,z) = h(y) + h(z)$. Then (\ref{8}) turns into
\begin{equation}\label{10}
d(x,y) \leq d(x,z) + h(y) + h(z).
\end{equation}
Since $\delta(x,y) = d(x,y)$, by (\ref{8}) and (\ref{2}) it follows that
\begin{gather*}
  %\label{11}
  d(x,y) \leq h(x) + h(y) = h(x) - h(z) + h(y) + h(z) \leq d(x,z) + h(y) + h(z) ,
\end{gather*}
which yields (\ref{10}). Now let two $\delta$'s be equal to the corresponding sums of the $h$'s. This case splits into the following ones
\begin{gather}
  \label{12}
  d(x,y) \leq h(x) + h(y) + 2 h(z), \\[.2cm] \nonumber h(x) \leq d(x,z) + h(z).
\end{gather}
The validity of the first line in (\ref{12}) follows by the fact that $d(x,y) \leq h(x) + h(y)$, see (\ref{8}). The validity of the second one follows by (\ref{2}). This completes the proof of (\ref{8}).

Let us turn now to proving the compactness, which is equivalent to the completness of $(X,\delta)$ and the total boundedness of $\delta$.
Let $\{x_l\}_{l\in \mathds{N}}$ and $\{K_n\}_{n\in \mathds{N}}$ be a $\delta$-Cauchy sequence and the ascending sequence as in (\ref{1}), respectively. Then either $\{x_l\}_{l\in \mathds{N}}\subset K_n$ for some $n$, or there exists a subsequence $\{x_{l_m}\}_{m\in \mathds{N}} \subset \{x_l\}_{l\in \mathds{N}}$ such that, for each $n$, there exists $m_n$ such that $x_{l_m}\in X\setminus K_{n}$ for all $m>m_n$.
In the former case,  there exists a subsequence $\{x_{l_p}\}_{p\in \mathds{N}} \subset \{x_l\}_{l\in \mathds{N}}$ $d$-convergent to a certain $x\in K_n$. By (\ref{7}) $\{x_{l_p}\}_{p\in \mathds{N}}$ is also $\delta$-convergent to $x$; hence, the whole $\{x_l\}_{l\in \mathds{N}}$ is $\delta$-convergent. In the latter case, $g(x_{l_m})\to 0$ as $m\to +\infty$, see (\ref{1}). The latter implies $h(x_{l_m})\to 0$, which  yields
\[
0 \leq \delta(x_{l_m}, x_0) \leq h(x_{l_m})+h(x_0) = h(x_{l_m}) \to 0.
\]
Hence, $\{x_{l_m}\}_{m\in \mathds{N}}$, and thus also $\{x_{l}\}_{l\in \mathds{N}}$ converge in $\delta$ to $x_0$. This yields the completness of $\delta$. To prove the total boundedness, we have to show that, for each $\varepsilon>0$, there exists a finite $D_\varepsilon\subset X$ such that $B^\delta_\varepsilon (x) \cap D_\varepsilon\neq \varnothing$ holding for each $x\in X$.
Here $B^\delta_\varepsilon (x) =\{y\in X: \delta (x,y)<\varepsilon\}$, which by (\ref{7}) contains the ball $B_\varepsilon (x)$. Let $\{r_n\}_{n\in \mathds{N}}$ be the sequence that appears in (\ref{1}).
For a given $\varepsilon>0$, find $n_\varepsilon \in \mathds{N}$ such that $r_n < \varepsilon$ whenever $n\geq n_\varepsilon$. By (\ref{1}) we then have that $h(x)<\varepsilon$ for all $x\in X\setminus K_{n_\varepsilon}$. Since $K_{n_\varepsilon}$ is compact in $(X,\mathcal{T})$, one finds a finite $C_\varepsilon\subset K_{n_\varepsilon}$ such that $B_\varepsilon (x)\cap C_\varepsilon\neq \varnothing$, holding for each $x\in K_{n_\varepsilon}$. Then the set in question is $D_\varepsilon = C_\varepsilon \cup\{x_0\}$. This completes the proof.
\end{proof}

\noindent

{\it Proof of Theorem \ref{2tm}.} In view of Lemma \ref{1lm}, it
remains to show that the spaces $(X, \mathcal{T}_{x_0})$ and $(Y,
\mathcal{T}')$ are homeomorphic. Here by $\mathcal{T}_{x_0}\subset
\mathcal{T}$ we mean the metric topology associated with $\delta$.
In is clear now that for any two disjoint $\delta$-convergent
sequences $\{x_k\}, \{x'_k\}\subset X$, the closures of $\{f(x_k)\}$
and $\{f(x'_k)\}$ in $(Y,\mathcal{T}')$ are disjoint, which means
that the map $f^{-1}: Z \to X$ can now be continuously extended to
the whole $Y$. As such, it turns into a homeomorphism, cf.
\cite[Theorem 7.7, page 19]{Bredon}, which yields the proof. \hfill
$\square$


\begin{thebibliography}{ll}

%\bibitem{AL1} J. M. Aarts, D. J. Lutzer, Pseudo-completness and the product of Baire spaces, Pacific J. Math. (1973) 48: 1--10.
%\bibitem{AA} J. M. Aarts, D. J. Lutzer, Completness Properties
%Designed for Recognizing Baire Spaces, Dissertationes Math., CXVI,
%Warszawa PWN, 1974.

\bibitem{Bes} A. Be\v{s}lagi\'c, Embedding cosmic spaces in Lusin spaces, Proc. Amer. Math. Soc. (1983) 89: 515--518.
\bibitem{BBR} L. Beznea, N. Boboc, M. R\"ockner, Quasi-regular Dirichlet forms and $L^p$-resolvents
on measurable spaces,
Potential Anal. (2006) 25: 269–282.

\bibitem{BK} D. Blount, M. Kouritzin, On convergence determining and separating
classes of functions, Stochastic Process. Appl., (2010) 120:
1898--=1907.

\bibitem{Bredon} G. E. Bredon, Topology and Geometry. Graduate Texts in Mathematics 139, Springer New York, 1993.
\bibitem{Cohn} D. L. Cohn, Measure Theory. Second edition. Birkh\"auser Advanced Texts: Basler Lehrb\"ucher.  Birkh\"auser/Springer,
 New York, 2013.ng

 \bibitem{Craul} H. Crauel, Random Probability Measures on Polish
 spaces. Stochastic Monographs, vol. 11, CRC Press, 2003.
 \bibitem{Dawson} D. A. Dawson, Measure-valued Markov Processes. {\'E}cole d'{\'E}t{\'e} de
Probabilit{\'e}s de Saint-Flour XXI--1991, 1--260, Lecture Notes in
Math., 1541, \emph{Springer}, Berlin, 1993.


\bibitem{Eng} R. Engelking, General Topology, 2nd edition, PWN, Warszawa, 1985.


\bibitem{HM} R. C. Haworth, R. A. MaCoy, Baire Spaces, Dissertationes Math., CXLI, Warszawa PWN, 1977.
%\bibitem{KR} Y. Kozitsky, M.  R\"ockner, A Markov process for an
%infinite interacting particle system in the continuum. Electron. J.
%Probab. (2021) 26: 1 -- 53.
%\bibitem{Li}Z. Li,  Measure-valued Branching Markov Processes, Probability and its
%Applications, \emph{Springer}, Heidelberg Dordrecht London New York, 2011.

\bibitem{Levi} S. Levi, A. Maitra, Borel measurable images of Polish
spacs, Proc. Amer. Math. Soc. (1984) 92: 98--102.

%\bibitem{MOR} Z. M. Ma, L. Overbeck, M. R\"ockner, Markov processes associated with semi-Dirichlet forms,
%Osaka J. Math. (1995) 32: 97--119.
\bibitem{Mandel} M. Mandelkern,  Metrization of the one-point
compactification. Proc. Amer. Math. Soc. (1989) 107: 1111--1115.

\bibitem{Nik} J. Niknejad, V. V. Tkachuk, L. Yengulalp,
Polish factorizations, cosmic spaces and domain representability.
Bull. Belg. Math. Soc. (2018) 25(3):  439--452.

%\bibitem{Palais}  R. S. Palais, When proper maps are closed, Proc. Amer. Math. Soc. (1970) 24: 835--836.

\bibitem{Taim} A.D. Taimanov, On extension of continuous mappings of topological spaces. Mat. Sbornik N.S. (1952) 31: 459--463.

\bibitem{PT} P. Ter\'an, Choquet theorem for random sets in Polish
spaces and beyond. In: Destercke, S., Denoeux, T., Gil, M.\'A.,
Grzegorzewski, P., Hryniewicz, O. (eds.) SMPS 2018. AISC, vol. 832,
pp. 208--215. Springer, Cham, 2019.

\end{thebibliography}
\end{document}